\documentclass[12pt]{amsart}

\usepackage{amsthm}
\usepackage{amsmath}
\usepackage{amsfonts,amssymb,latexsym}
\usepackage[all,arc,curve,color,frame,line]{xy}
\usepackage{epsfig,epsf}
\usepackage{color}
\usepackage{latexsym,amssymb,enumerate, amsmath}

\newtheorem*{theoremA}{Theorem A}
\newtheorem*{theoremB}{Theorem B}
\newtheorem*{theoremC}{Theorem C}
\newtheorem*{theoremD}{Theorem D}
\newtheorem*{theoremE}{Theorem E}
\newtheorem*{theoremF}{Theorem F}

\newtheorem{Thm}{Theorem}
\newtheorem{Lemma}{Lemma}
\newtheorem{cor}{Corollary}

\newtheorem{Proposition}{Proposition}

\title[Subsets of orderable groups]
{On the structure of subsets\\ of an orderable group \\ with some small doubling properties}

\author[Freiman, Herzog, Longobardi, Maj, Plagne, Robinson \and Stanchescu ]
{G. A. Freiman, M. Herzog, P. Longobardi, M. Maj, \\  A. Plagne, D. J. S. Robinson \and Y.V. Stanchescu}

\address{Gregory A. Freiman, Marcel Herzog
\newline 
\indent School of Mathematical Sciences, Tel Aviv University,
Tel Aviv 69978, Israel}
\email{~grisha@post.tau.ac.il}
\email{~herzogm@post.tau.ac.il}

\address{Patrizia Longobardi, Mercede Maj
\newline
\indent Dipartimento di Matematica, Universita' di Salerno, Via Giovanni Paolo II, 132, 84084 Fisciano (Salerno), Italy}
\email{~plongobardi@unisa.it}
\email{~mmaj@unisa.it}

\address{Alain Plagne
\newline
\indent Centre de Math\'ematiques Laurent Schwartz, \'Ecole polytechnique, 91128 Palaiseau Cedex,France}
\email{~plagne@math.polytechnique.fr}

\address{Derek J. S. Robinson
\newline 
\indent Department of Mathematics, University of Illinois at Urbana-Champaign, Urbana, IL 61801, USA}
\email{~dsrobins@illinois.edu}

\address{Yonutz V. Stanchescu
\newline
\indent Afeka Academic College, Tel Aviv 69107, Israel
\newline 
and
\newline
\indent The Open University of Israel, Raanana 43107, Israel}
\email{~yonis@afeka.ac.il {\rm \it and} ~ionut@openu.ac.il}


\begin{document}
\maketitle

\section{Introduction}

Let $G$ denote an arbitrary group (multiplicatively written). If $S$ is a subset of $G$,
we define its {\em square} $S^2$ by the formula
$$
S^2=\{x_1x_2 \mid x_1,x_2\in S\}.
$$
In the abelian context, $G$ will usually be additively written and we 
shall rather speak of {\em sumsets} and specifically of the {\em double} of $S$, namely
$$
2S=\{x_1+x_2 \mid x_1,x_2\in S\}.
$$

Here, we are concerned with the following general problem: for two real numbers $\alpha \geq 1$ 
and $\beta$, determine the {\em structure} of $S$ if $S$ is a finite subset 
of a group $G$ satisfying an inequality on cardinalities of the form
$$
|S^2| \leq \alpha|S| + \beta
$$
when $\alpha$ is small and $|S|$ is typically large. 

Problems of this kind are called {\em inverse problems} of {\em small doubling} type in 
additive number theory. 
The coefficient $\alpha$ (or more precisely the ratio $|S^2|/|S|$) is called the 
{\em doubling coefficient} of $S$.
This type of problems became the most central issue in additive combinatorics.
Inverse problems of small doubling type have been first investigated 
by G.A. Freiman very precisely in the additive group of the integers 
(see \cite{F1}, \cite{F2}, \cite{F3}, \cite{F4}) and by many other 
authors in general abelian groups, starting with M. Kneser  \cite{K}
(see, for example, \cite{Ham},  \cite{LS}, \cite{Bilu}, \cite{Ru}, \cite{GR}).  
More recently, small doubling problems in non-necessarily abelian groups 
have been also studied, see \cite{G},  \cite{Sa} and \cite{BGT} for 
recent surveys on these problems and  \cite{NA} and \cite{TV} 
for two important books on the subject.

It is easy to prove that if $S$ is a  finite subset of $\mathbb{Z}$, then
$$
2|S|-1 \leq  |2S|  \leq \frac{|S|(|S|+1)}2.
$$
Moreover $|2S| = 2|S|-1$ if and only if  $S$ is a (finite) arithmetic progression, 
that is, a set of the form
$$
\{a, a+q, a+2q, \dots, a+(t-1)q\}
$$
where $a$, $q$ and $t$ are three integers, $t \geq 1$, $q \geq 0$. 
The parameter $t$ is called the {\em size} of the arithmetic progression and $q$ its 
{\em difference} 
(we shall use {\em ratio} in the multiplicative notation).
In the articles \cite{F1} and \cite{F2}, G.A. Freiman proved the following more general  
results. The first result is referred 
to as the $3k-4$ theorem.
  
\begin{theoremA}
Let $S$ be a finite set of integers with at least three elements. 
If  $|2S|\leq 3|S|-4$, then $S$ is contained in an arithmetic progression of size $|2S|-|S|+1 \leq 2 |S|-3$. .
\end{theoremA}

The second result goes one step further. It is called the $3k-3$ theorem.

\begin{theoremB}
Let $S$ be a finite set of integers with at least $3$ elements.
If $|2S| = 3|S|-3$, then one of the following holds:
\begin{itemize}
\item[(i)] $S$ is a subset of an arithmetic progression of size at
most $2|S| - 1$,
\item[(ii)] $S$ is 
the union of two arithmetic progressions with the same difference,
\item[(iii)] $|S| = 6$ and $S$ is Freiman-isomorphic to the set $K_6$, where			 $$
K_6 = \{(0,0), (1,0), (2,0), (0,1), (0,2), (1,1)\}.
$$
\end{itemize}
\end{theoremB}

Recall that two sets are Freiman isomorphic if they behave with respect to addition in the same way 
(there is a one-to-one 
correspondence between the two sets as well as between their two sumsets). 
Notice that translations and dilations are transformations which send a set on 
another set which is Freiman 
isomorphic to it. For instance, in case (iii) of the above-stated theorem, such a set 
$S$ of integers is, 
up to translation and dilation, of the form 
$$
\{ 0,1,2, u, u+1, 2u \}
$$
for an integer $u \geq 5$.

Freiman also investigated in these papers the exact structure of subsets 
$S$ of the additive group $\mathbb{Z}$ if $|2S| = 3|S|-2$ (Theorem $3k-2$) and more generally 
of small doubling sumsets of $\mathbb{Z}^d$ (see for instance \cite{S2}).

Since then, far-reaching generalizations have been obtained, see \cite{BGT}. Although very powerful and general, 
these last results are not very precise.

In papers \cite{FHLM}, \cite{FHLMS2}, \cite{FHLMS3}, \cite{FHLMS1} and \cite{FHLMPS}, we started the precise investigation of 
small doubling problems for subsets of an ordered group. We recall that if $G$ is a group 
and $\leq$  is a total order relation defined
on the set $G$, then $(G, \leq)$ is an {\em ordered group} if for all $a,b,x,y\in G$ 
the inequality $a\leq b$ implies that $xay\leq xby$, and  a group $G$ is {\em orderable} 
if there exists an order $\leq$  on the set $G$ such that $(G, \leq)$ is an ordered group. 
Obviously the group of  integers with the usual ordering is an ordered group. More generally, 
it is possible to prove that a nilpotent group is orderable if and only if it is torsion-free 
(see, for example, \cite{MA} or \cite{N}).

Extending Freiman's results, we proved in \cite{FHLM} (Corollary 1.4) 
the following theorem, which is analogous to the $3k-4$ theorem.

\begin{theoremC} 
Let $G$ 
be an orderable group and let  
$S$ 
be a finite subset of $G$ with at least three elements. 
If 
$|S^2| \leq 3|S|-4$, then $\langle S \rangle$ is abelian. Moreover, there exists elements 
$x_1$ and $g$ in $G$,
such that $gx_1=x_1g$ and $S$ is a subset of
$$
\{x_1, x_1g, x_1g^2, \dots, x_1g^{u}\}
$$
where $u=|S^2|-|S|$. In other words, $S$ is contained in an abelian geometric 
progression of size 
$|S^2|-|S|+1$.
\end{theoremC}

Furthermore, in the case when $|S^2|  = 3|S|-3$, we obtained  the following initial result (this is our Theorem 1.3 in \cite{FHLM})

\begin{theoremD} 
Let $G$ be an orderable group
and let $S$ 
be a finite subset of $G$ of size at least $3$. 
If $|S^2| \leq 3|S|-3$, then $\langle S \rangle$ is abelian.
\end{theoremD}

Using  results of Freiman and Stanchescu, we proved in \cite{FHLMPS} (Theorem 1)
the following theorem concerning the more precise structure of such small doubling sets $S$.

\begin{theoremE} 
Let $G$ be an orderable group 
and let $S$ be a finite subset of $G$ of size at least $3$ 
satisfying $|S^2| \leq 3|S|-3$. Then $\langle S \rangle $ is abelian and at most $3$-generated.

Moreover, if $|S|\geq 11$, then one of the following two possibilities occurs:
\begin{itemize}
\item[(i)] $S$ is a subset of a geometric progression of length at most $2|S|-1$,
\item[(ii)] $S$ is a  union of two geometric progressions with the same ratio.

\end{itemize}
\end{theoremE}

We also proved the following result concerning the  structure of $\langle S \rangle$ 
in the case when $|S^2| = 3|S|-2$ (see  \cite{FHLMPS}, Theorems 2 and 5). To state it,
recall the important notation
$$
[a, b] =a^{-1}b^{-1}ab \quad \text{ and } \quad a^b=b^{-1}ab.
$$

\begin{theoremF}  Let $G$ be an orderable group 
and let $S$ be a finite subset of $G$ of size at least $3$.

If $|S^2| = 3|S|-2$, then one of the following holds:
\begin{itemize}
\item[(i)] $ \langle S \rangle$ is abelian and at most 4-generated,
\item[(ii)]  $\langle S \rangle = \langle a, b\rangle$, with $[a,b]\neq 1$ and 
$[[a,b], a] = [[a,b], b] = 1$. In particular, $\langle S \rangle $ is nilpotent of class 2,
\item[(iii)]  $\langle S \rangle = \langle a, b\rangle$, with $ a^{b^2} = aa^b$ and $[a, a^b] = 1$,
\item[(iv)]  $\langle S \rangle = \langle a, b\rangle$, with $a^b = a^2$. In particular, 
$\langle S \rangle $ is a quotient of the Baumslag-Solitar group B(1,2),
\item[(v)] $\langle S \rangle = \langle a \rangle \times   \langle c, b \rangle$, with
$c^b = c^2$,  $|S| = 4$ and  $S = \{a, ac, ac^2, y\}$, where $y$ is a suitable element of $\langle c, b \rangle$.
\end{itemize}
\end{theoremF}

If $\langle S \rangle$ is an abelian group which is at most 4-generated and $|S^2| = 3|S|-2$, 
then the structure  of $S$ can be deduced from  previous results of Freiman and 
Stanchescu (see \cite{FHLMPS}, Theorem 2). 
The aim of this paper is to go one step further in the non-abelian case. 
Namely, we shall present a complete description of the structure of $S$ if $S$ is a finite 
subset of an orderable group $G$ with $|S^2| = 3|S|-2$ and $\langle S \rangle$ is non-abelian.
The following result will be our main theorem.

 \begin{Thm}
 \label{1ertheo}
Let $(G, \leq)$ be an ordered group and let $S$  be a finite subset of $G$ of size at least $4$. 
We assume that $|S^2| = 3|S|-2$ and that $\langle S \rangle$ is non-abelian. 

Then one of the following holds:
\begin{itemize}
\item[(i)] $S=\{a, ac, \dots, ac^i, b, bc, \dots, bc^j \},$  where   
$[a, c] = [ b,c] = 1, c > 1$ and either $ab = bac$ or $ba = abc$,
 \item[(ii)]  $ S = \{x, xc, xc^2, \dots, xc^{k-1}\},$ where $ c > 1$ and either  
$ c^x = c^2$  or  $(c^2)^x = c$,
\item[(iii)]  $ |S| = 4$  and  the  structure of  $S$  is of one of the following types:
\begin{itemize}
\item[(a)]  either $S=\{x,xc,xc^x,xc^{x^2}\}$ or $S=\{x^{-1},x^{-1}c,x^{-1}c^x,x^{-1}c^{x^2}\}$,
with $c>1$ and $c^{x^2}=cc^x=c^xc$,
\item[(b)] $S=\{1,c,c^2,x\}$, where either $c^x=c^2$ or $(c^2)^x=c$,
\item[(c)] $S=\{x,xc,xc^2,y\}$, where $[c,x]=1$ and either $[x,y]=c=(c^2)^y$ or $[y,x]=c^2=c^y$,
\item[(d)] $S=\{x,xc,xc^2,y\}$, where $[c,x]=[x,y]=1$ and either $c^y=c^2$ or $(c^2)^y=c$.
\end{itemize}
\end{itemize}
Moreover, if either (i) or (ii) or case (b) of (iii) holds, then $|S^2| = 3|S|-2$.
\end{Thm}

This  paper is organized as follows. 
In Section \ref{sec2}, we record  some useful results from \cite{FHLMPS} and \cite{FHLMS1}.
In Section  \ref{sec3}, we first study the group  $\langle a, b \ | a^{b^2} = aa^b, [a, a^b] = 1\rangle$. 
Then we investigate  the structure of $S$ if $\langle S \rangle =  \langle a, b \ | 
a^{b^2} = aa^b, [a, a^b] = 1\rangle$ and $|S^2| = 3|S|-2$. We prove in Theorem \ref{2emetheo}
that these assumptions
imply that $|S|\leq 4$  and we present a complete description of $S$ if $|S| = 4$ and $|S^2| = 10.$
In Section \ref{sec4} we record some basic results concerning the Baumslag-Solitar group $B(1,2)$,
mainly from \cite{FHLMS3}. 
Then we present in Theorem \ref{3emetheo}  a complete description of subsets $S$ of $B(1,2)$ 
satisfying $\langle S \rangle = G$ and $|S^2| = 3|S| -2.$ 
Notice that small doubling 
problems in the Baumslag-Solitar groups have been also studied  in \cite{FHLMS2}.
Finally, in Section \ref{sec5} we prove Theorem \ref{1ertheo}.
We refer to books \cite{BR}, \cite{F3} \cite{R} and \cite{Rob81} for notation and definitions.

\section{Some useful results}
\label{sec2}

Here we collect the following useful results from \cite{FHLMPS} and \cite{FHLMS1}.

\begin{Proposition} [see \cite{FHLMPS}, Lemma 4]
\label{prop1}
Let $G$ be an orderable  group 
and let $S$ be a finite subset of $G$ of size at least $2$. 
If $T$ denotes $S \setminus \{ \max S \}$, then

either $\langle S \rangle$ is a 2-generated abelian group, or $|T^2| \leq |S^2|-3$.
\end{Proposition}

In the next proposition, we use the notation $\dot\cup$ for a disjoint union.

\begin{Proposition} [see \cite{FHLMPS}, Proposition 3]
\label{prop2}
Let $G$ be an orderable group 
and let $S$ be a finite subset of $G$ of size at least $3$ 
satisfying $|S^2| = 3|S|-2$. Suppose that $S = T \dot\cup \{y\}$, 
where $\langle T \rangle$ is abelian. Then either $\langle S \rangle$ is abelian, or 
$$S = \{x, xc, \dots, xc^{k-2}, y \}$$ 
and one of the following holds: 
\begin{itemize}
\item[(i)] $[c,x] = [c, y] = 1$ and either $\ [x, y] = c$ or $[y,x] = c$,
\item[(ii)] $|S| = 4$, $[c,x]=1$  and either $[x, y] = c=(c^2)^y$ or 
$[y, x] = c^2=c^y$,
\item[(iii)] $|S| = 4$, $[c, x]=[x, y] = 1$ and either $c^y = c^2$  or $(c^2)^y = c$.
\end{itemize}
\end{Proposition}

The following two propositions deal with the case when $|S|=3$.

\begin{Proposition} [see \cite{FHLMPS}, Proposition 1]
\label{prop3}
Let $(G, \leq)$ be an ordered group and let $x_1, x_2, x_3$ be three elements of $G$ such that $x_1 < x_2 < x_3$.
Let $S = \{x_1, x_2, x_3\}$.
Suppose that  $\langle S \rangle$ is non-abelian and
either $x_1x_2 = x_2x_1$ or $x_2x_3 = x_3x_2$.  
Then $|S^2| = 7$ if and only if one of the following holds:
\begin{itemize}
\item[(i)]
$S \cap Z(\langle S \rangle) \not= \emptyset$,
\item[(ii)] $S$ is of the form $\{a, a^b, b\}$, where $aa^b = a^ba$.
\end{itemize}
\end{Proposition}

\begin{Proposition} [see \cite{FHLMPS}, Proposition 2]
\label{prop4}

Let $(G, \leq)$ be an ordered group, 
and let $x_1, x_2, x_3$ be three elements of $G$ such that $x_1 < x_2 < x_3$.
Let $S = \{x_1, x_2, x_3\}$
and assume that $x_1x_2 \not= x_2x_1$ and $x_2x_3 \not = x_3x_2$.
If $|S^2| = 7$, then one of the following statements holds:
\begin{itemize}
\item[(i)] either $S = \{x, xc, xc^x\}\ or\ S = \{x^{-1}, x^{-1}c, x^{-1}c^x\}$, with $c > 1$, 
$c \in G^{\prime}$ and $c^{x^2} = cc^x=c^xc$,
\item[(ii)] either $S = \{x, xc, xcc^x\}$ or $S = \{x^{-1}, x^{-1}c, x^{-1}cc^x\}$, 
with $c > 1$, 
$c \in G^{\prime}$ and $c^{x^2} = cc^x= c^xc$,
\item[(iii)]  $S = \{x, xc, xc^2\}$, with either $c^x = c^2$ or $(c^2)^x = c$.
\end{itemize}
\end{Proposition}

Next proposition considers the case when $|S|=4$.

\begin{Proposition} [see \cite{FHLMPS}, Lemma 6] 
\label{prop5}
Let $(G, \leq)$ be an ordered group 
and let $x_1, x_2, x_3, x_4$ be four elements of $G$ such that $x_1 < x_2 < x_3<x_4$.
Let $S = \{x_1, x_2, x_3,x_4\}$ and suppose that $|S^2| = 10$. If $x_2x_3 = x_3x_2$, then either 
$\langle x_2, x_3, x_4 \rangle$ or $\langle x_1, x_2, x_3 \rangle$ is abelian.
\end{Proposition}

The final proposition which we shall need deals with the case of nilpotent groups of class $2$.

\begin{Proposition} [see \cite{FHLMS1}, Theorem 3.2] 
\label{prop6}
Let $G$ be an orderable 
nilpotent group of class $2$ and let $S$ be a finite subset of $G$ of size at least $4$,  
such that $\langle S \rangle $ is non-abelian. Then $|S^2|=3|S|-2$ if and only if $S=\{a,ac,\dots,ac^i,b,bc,\dots,bc^j\}$,
with 
$c>1$ and either $ab=bac$ or $ba=abc$. 
\end{Proposition}

\section{Subsets of the group $G = \langle a, b \ | \ a^{b^2} = aa^b,\ aa^b = a^ba \rangle$ }
\label{sec3}

In this section we shall prove the following theorem.

\begin{Thm}
\label{2emetheo}
Let 
$$
G = \langle a, b \ | \ a^{b^2} = aa^b,\ aa^b = a^ba \rangle.
$$ 

\begin{itemize}
\item[(i)] The group $G$ is orderable. 

\noindent  Let $\leq$ be a total order on $G$.

\item[(ii)] Let $S$ be a finite subset of $G$ satisfying $\langle S \rangle = G$ and $|S^2| = 3|S|-2.$
Then $|S| \leq 4.$
\item[(iii)] Moreover, in the case when $|S| = 4$, then either $S = \{x, xc, xc^x, xc^{x^2}\}$ or $S = \{x^{-1}, x^{-1}c, x^{-1}c^x, x^{-1}c^{x^2}\}$, 
with $c>1$  and $ c^{x^2} = cc^x = c^xc$.
\end{itemize}
\end{Thm}

Throughout this section, we shall denote by $G$ the following group: 
$$
G =  \langle a, b \ | \ a^{b^2} = aa^b,\ aa^b = a^ba \rangle.
$$  
We begin with some remarks concerning $G$.

Define $H=(\langle u \rangle \times \langle v
\rangle)\rtimes\langle t \rangle$, where $t,u,v$ have infinite order
and $u^t=v$, $v^t=uv$. Then the defining relations of $G$ hold in $H$,
namely $[u,u^t]=1$ and $u^{t^2}=uu^t$ and by von Dyck's theorem (Theorem 2.2.1 in \cite{Rob81})
there is an epimorphism $\theta : G\to H$ with $a^{\theta}=u$ and $b^{\theta}=t$.
Since $\ker \theta =1$, it follows that 
$$
G = (\langle a \rangle \times \langle a^b \rangle)\rtimes\langle b \rangle, \text{ with }  a^{b^2} = aa^b.
$$ 
Thus $G ^{\prime} = \langle a \rangle \times \langle a^b \rangle$ and $G$ is a polycyclic metabelian group. 

We have $a^{b^2} = aa^b,\ a^{b^3} = a(a^b)^2$ and it is easy to see,
by induction on $n$, that
\begin{equation}
\label{star} 
a^{b^n} = a^{f_{n-1}}(a^b)^{f_n} \text{ for any } n \in \mathbb{N}, 
\end{equation}
where $(f_n)_{n \in \mathbb{N}_0}$ is the Fibonacci sequence  defined in the standard way  by induction by
$f_0 = 0, f_1 = 1$ and $ f_{n+1} = f_n+f_{n-1}$ for  $n>0.$ In particular, $f_2=1$ and 
$a^{b^2}=a^{b+1}$.

Furthermore, recall that a group is called an {\em $R^{\star}$-group} if
$$
g^{x_1}\dots g^{x_n}=e\ \text{implies}\ g=e\ \text{for all}\ n\in \mathbb{N}\ 
\text{and all}\ g,x_1,\dots,x_n\in G.
$$
Since metabelian $R^{\star}$-groups are orderable (see \cite{BR}, Theorem 4.2.2), in
order to prove that $G$ is orderable it suffices to show that it is an $R^{\star}$-group. This is what we do now.

\begin{Proposition} 
$G$ is an $R^\star$-group and hence it is orderable.
\end{Proposition}

\begin{proof} 
It suffices to show that, if $n\in \mathbb{N}$,  $k, u, v \in \mathbb{Z}$ and $g_i\in G$, then
$$
(b^ka^u(a^b)^v)^{g_1}\cdots (b^ka^u(a^b)^v)^{g_r} = 1
$$
implies that $ b^ka^u(a^b)^v = 1.$

First we notice that by working mod $G^{\prime}$ it follows that $k=0$.  
Hence we may assume that 
\begin{equation}
\label{starstar} 
(a^u(a^b)^v)^{b^{n_1}}(a^u(a^b)^v)^{b^{n_2}} \cdots (a^u(a^b)^v)^{b^{n_r}}= 1,
\end{equation}
where $n_i \geq 2$, applying  conjugation by a power of $b$, if necessary. 
Since by \eqref{star} 
$$
a^{b^{n_1}+b^{n_2}+\cdots+b^{n_r}} = a^{c+db}
$$
where $c, d $ are positive integers, relation \eqref{starstar} now becomes
$$
(a^u)^{c+db}(a^v)^{cb+db^2} = 1.
$$
Equivalently
$$
a^{uc+udb}a^{vcb+vd(b+1)} = 1,
$$
or 
$$
a^{(uc+vd)+(ud+vc+vd)b} = 1.
$$

Hence we have
$$
uc+vd =  0 \quad \text{ and } \quad ud+vc+vd =  0
$$

which implies that
$$ud^2+(-uc)c+d(-uc) = 0.$$
If $u = 0$, then $vd = 0$ and $v = 0$, so we are done. Hence we may assume 
that $u \not= 0$. It follows that 
$$
d^2-dc-c^2 = 0.
$$ 
Therefore $d/c$, which is rational, must satisfy 
$$
\frac{d}{c}= \frac{1+\sqrt{5}}{2},
$$ 

a contradiction.
\end{proof}

\begin{cor}
Part (i) of Theorem \ref{2emetheo} holds.
\end{cor}

Moreover, the following property holds in $G$.

\begin{Proposition} 
\label{prop8}
We have $C_{G^{\prime}}(g) = \{1\}$ for any $g \in G \setminus G^{\prime}$. 
In particular $Z(G) = \{1\}$.
\end{Proposition}

\begin{proof} 
Let $g \in G \setminus G^{\prime}$. Then $g = db^n$, where $d \in G^{\prime}$ 
and $n$ is an integer different from $0$. 
Since $d$ centralizes $G^{\prime}$, it suffices to show that $C_{G^{\prime}}(b^n) = \{1\}$,
where $n$ is a positive integer different from $0$.  

Denote the element $a^u(a^b)^v$ of  $G^{\prime}$ by $(u,v)$, where $u,v$ are integers.
Then $b$ acts on $(u,v)$ by conjugation via the following function:
$$
(u,v)^b=(0\cdot u+1\cdot v,1\cdot u+1\cdot v)=(u,v)B
$$
where 
$$
B=\begin{bmatrix} 
0 & 1\\
1 & 1
\end{bmatrix} .
$$
Thus, by induction, $b^n$ acts by conjugation on $(u,v)$ via multiplication by:
$$
B^n=\begin{bmatrix} f_{n-1} & f_n\\
f_n & f_{n+1}\end{bmatrix} ,
$$
where $(f_m)_{m \in \mathbb{N}_0}$ is the Fibonacci sequence.

If $b^n$ centralizes a non-trivial element of $G^{\prime}$, then $1$ would be
an eigenvalue of $B^n$. But the characteristic polynomial of $B^n$ is
$x^2-(f_{n-1}+f_{n+1})x+(f_{n-1}f_{n+1}-f_n^2)$, which has roots 
$(f_{n-1}+f_{n+1} \pm \sqrt{5}f_n)/2$, which are irrational for $n>0$.
Hence $C_{G^{\prime}}(b^n) = \{1\}$, as required. 
\end{proof}

Now, let $\leq$ be a total order in $G$ such that $(G, \leq)$ is an ordered group. 
Let $S = \{x_1, x_2, \dots, x_k\}$ be a subset of $G$ of size $k$ and suppose 
that $x_1 < x_2 < \cdots < x_k$. We wish to study the structure of $S$ if $|S^2| = 3|S|-2.$ 
We begin with the case $k = 3$.

\begin{Proposition} 
\label{prop9}
Let $S \subseteq G$, with $\langle S \rangle = G$. Suppose that $|S| = 3$. 
Then $|S^2| = 7$ if and only if one of the following holds:
\begin{itemize}
\item[(i)] $S = \{1, x, y\}$, with  $xy \not= yx$,
\item[(ii)] $S = \{c, c^t, t\}$, with  $cc^t = c^tc$,
\item[(iii)] either $S = \{t, tc, tc^t\}$ or $S = \{t, tc, tc^{t^2}\}$, 
where $c \in G^{\prime}$, 
$c > 1$ and  $c^{t^2} = cc^t = c^tc$,
\item[(iv)] either $S = \{t^{-1}, t^{-1}c, t^{-1}c^t\}$ or $S = \{t^{-1}, t^{-1}c, 
t^{-1}c^{t^2}\}$, 
where $c \in G^{\prime}$, $c > 1$ and $c^{t^2} = cc^t = c^tc$.
\end{itemize}
Furthermore, in any case,  either $S \cap G^{\prime} \subseteq \{1\}$ or $|S \cap G^{\prime} | = 2$.
\end{Proposition}

\begin{proof}
Suppose that $S = \{x_1, x_2, x_3\}$ and $|S^2| = 7$. If either $x_1x_2 = x_2x_1$
or $x_2x_3 = x_3x_2$ holds, then either (i) or (ii) holds by Proposition \ref{prop3}, since $Z(G) = \{1\}$.  

If $x_1x_2 \not= x_2x_1$ and $x_2x_3 \not= x_3x_2$, then either (iii) or (iv) holds 
by Proposition \ref{prop4}, 
since the relation $c^x = c^2$ is impossible in a torsion-free polycyclic group (otherwise 
there would be an infinite chain of subgroups of $G$: 
$ \langle c \rangle \subset \langle c^{x^{-1}}\rangle  
\subset  \langle c^{x^{-2}}\rangle \subset\cdots \subset  
\langle c^{x^{-n}}\rangle \subset \cdots$).

In particular, if (i) holds, then either  $S \cap G^{\prime}=\{1\}$ or $S \cap G^{\prime}=\{1,z\}$
with $z\in \{x,y\}$, since $G^{\prime}$ is abelian and $G$ is non-abelian. 
In case (ii), $c^t=cd$ with 
$d\in G^{\prime}\setminus \{1\}$ and $cc^t = c^tc$ implies that $cd=dc$. Hence, by 
Proposition 8, $c\in S \cap G^{\prime}$ and since $c^t\neq c$, it follows that 
$t\notin  S\cap G^{\prime}$. Thus $S \cap G^{\prime}=\{c,c^t\}$. In cases (iii) and (iv)
$c,c^t,c^{t^2}\in G^{\prime}$, so $t\notin G^{\prime}$ since $G$ is non-abelian.
Hence $S \cap G^{\prime}=\emptyset $ in these cases. Thus in all cases
either $S \cap G^{\prime} \subseteq \{1\}$ or $|S \cap G^{\prime} | = 2$ holds.
\medskip

A direct calculation proves the converse of the main statement.
\end{proof}

Now we study the structure of $S$ if  $|S| = 4$ and $|S^2| = 10$. We begin with the 
following four lemmas.
 
\begin{Lemma} 
\label{lemma1}
Let $S = \{t, tc, tc^t, x_4\} \subseteq G$,  with $c^{t^2} = c^tc = cc^t$ and 
$t < tc < tc^t < x_4$. Then $|S^2| = 10$ if and only if $x_4 = tc^{t^2}$.
\end{Lemma}

\begin{proof}
Suppose that $|S^2|= 10$.
Write $x_1 = t, x_2 = tc, x_3 =  tc^t$ and $ T = \{x_1, x_2 , x_3\}$. Then 
$$T^2 = \{t^2, t^2c, t^2c^t, t^2cc^t, t^2c^tc^t, t^2c^2c^t, t^2c(c^t)^2\}.$$
Notice that $c=(c^t)^{-1}(c^t)^t$, implying that $c$ and $c^t$ belong to $G^{\prime}$.
Moreover, $t\notin G^{\prime}$ since $c^t\neq c$, and $c>1$ since $t<tc$.

Obviously $x_3x_4, x_4^2 \notin T^2$ because of the ordering. Hence one of the elements 
$x_1x_4, x_2x_4$ belongs to $ T^2$, since $|T^2| = 7$ and $|S^2| = 10$. Thus $x_4 = td$
for some 
$d \in G^{\prime}$. If $x_3x_4 = x_4x_3$, then $t(d(c^t)^{-1})=((c^t)^{-1}d)t=(d(c^t)^{-1})t$
and $d=c^t$ by Proposition 8. But then $x_3=x_4$, a contradiction. Hence $x_3x_4\neq x_4x_3$
and
$$
S^2 = T^2 \dot\cup \{x_3x_4, x_4x_3, x_4^2\}.
$$
Notice that
$$
x_1x_4 = t^2d,\quad x_2x_4 = t^2c^td,\quad x_3x_4 = t^2cc^td, 
$$
$$
x_4x_1 = t^2d^t,\quad  x_4x_2 = t^2d^tc \quad \text{ and }\quad x_4x_3 = t^2d^tc^t.
$$
 
If $x_2x_4 \notin T^2$, then the only possibility is $x_2x_4 = x_4x_3$, thus $t^2c^td = t^2d^tc^t$ 
and $d^t = d$, a contradiction. Therefore $x_2x_4 \in T^2$ and the only possibilities are 
$d = c^2$ or $d = cc^t$, since $c^t < d$. If $d = c^2$, then $x_1x_4 = t^2c^2 \notin T^2$ 
and the only possibility is $x_1x_4 = x_4x_3$, yielding $d = d^tc^t$ and $c^2 = (c^t)^3$, 
a contradiction since $c < c^t$. Therefore the only remaining possibility is $d = cc^t$ 
and $x_4 = tc^{t^2}$, as required.

A direct calculation proves the converse.
\end{proof}

\begin{Lemma} 
\label{lemma2}
Let $S = \{t^{-1}, t^{-1}c, t^{-1}c^t, x_4\} \subseteq G$,  
with $c^{t^2} = c^tc = cc^t$ and $t^{-1} < t^{-1}c < t^{-1}c^t < x_4$. 
Then $|S^2| = 10$ if and only if $x_4 = t^{-1}c^{t^2}$.
\end{Lemma}

\begin{proof}
Suppose that $|S^2| = 10.$ 
As in the previous lemma, we may assume that  $x_4 = t^{-1}d$, where $d \in G^{\prime}$. Notice 
also that $c,c^t\in G^{\prime}$ and  $ G^{\prime}$ is abelian. Write 
$$\bar{S} = \{t, tc, tc^t, td\}.$$
Then 
\begin{eqnarray*}
|\bar{S^2}| & = & |t \{1, c, c^t, d\} t \{1, c, c^t, d\}| \\ 
		 & = & |t^2\{1, c, c^t, d\}^t \{1,c,c^t, d\}| \\
		 & = & |\{1, c, c^t, d\} \{1,c,c^t, d\}^{t^{-1}}| \\
		 & = & |\{1, c, c^t, d\}^{t^{-1}} \{1,c,c^t, d\}| \\
		& = & |t^{-2} \{1, c, c^t, d\}^{t^{-1}}\{1,c,c^t, d\}| \\
		& = & |S^2|.
\end{eqnarray*}
Therefore $|\bar{S}^2|=10$ and $d = cc^t=c^{t^2}$ by the previous lemma.

A direct calculation proves the converse.
\end{proof}

\begin{Lemma} 
\label{lemma3}
Let $S = \{t, tc, tc^{t^2}, x_4\} \subseteq G$,  with $c^{t^2} = c^tc = cc^t$ 
and $t < tc < tc^{t^2} < x_4$. Then $|S^2| > 10$.
\end{Lemma}

\begin{proof}
Write $x_1 = t, x_2 = tc, x_3 =  tc^{t^2}$ and $T = \{x_1, x_2, x_3\}$. Then 
$$T^2 = \{t^2, t^2c, t^2c^t, t^2cc^t, t^2(c^t)^2c, t^2c^2(c^t)^2, t^2c^2(c^t)^3\}$$
and, as in Lemma 1, $c>1$, $c,c^t\in G^{\prime}$ and $t\notin G^{\prime}$.
Thus, by Proposition 8, $\langle S\rangle $ is non-abelian and by Theorem D
$|S^2|\geq 3|S|-2=10$. So it suffices to assume that $|S^2|=10$ and to reach a contradiction.
 
Obviously $x_3x_4, x_4^2 \notin T^2$ because of the ordering. Hence one of the elements 
$x_1x_4, x_2x_4$ belongs to $T^2$, since $|T^2| = 7$ and $|S^2| = 10$. It follows, 
as in Lemma \ref{lemma1}, that $x_4 = td$ for some 
$d \in G^{\prime}$ and $x_3x_4 \not= x_4x_3$. Thus
$$
S^2 = T^2 \dot\cup \{x_3x_4, x_4x_3, x_4^2\}
$$
and
$$
x_1x_4 = t^2d, \quad x_2x_4 = t^2c^td, \quad x_3x_4 = t^2(c^t)^2cd,
$$
$$ 
x_4x_1 = t^2d^t, \quad x_4x_2 = t^2d^tc  \quad \text{ and } \quad x_4x_3 = t^2d^tc^tc.
$$ 
If $x_2x_4 \notin T^2$, then the only possibility is $x_2x_4 = x_4x_3$. 
But then $tctd=tdtc^{t^2}$ and $c^td=(c^td)^t$, in contradiction to Proposition \ref{prop8}.
Thus $x_2x_4 \in T^2$, which implies that either $d = c^2c^t$ or $d = c^2(c^t)^2$,
since $c^t<d$.
 
In the first case $x_1x_4 = t^2c^2c^t \notin T^2$, thus $x_1x_4 = x_4x_3$ and $d = d^tc^tc$.
But $d=c^2c^t$, so $d^t=c$ and $c=d^t=(c^t)^2c^{t^2}=(c^t)^3c$ since $c^{t^2}=c^tc$, 
a contradiction. 

If, on the other hand, $d = c^2(c^t)^2$, then $x_4x_1 = t^2c^2(c^t)^4 \notin T^2$, 
so $x_4x_1 = x_3x_4$ and $d^t=c(c^t)^2d$. Thus
$c^2(c^t)^4 = c^3(c^t)^4$, again a contradiction. 

We have reached a contradiction in all possible cases and our Lemma is proved.
\end{proof}

\begin{Lemma} 
\label{lemma4}
Let $S = \{t^{-1}, t^{-1}c, t^{-1}c^{t^2}, x_4\} \subseteq G$,  
with $c^{t^2} = c^tc = cc^t$ and $t^{-1} < t^{-1}c < t^{-1}c^t < x_4$. Then $|S^2| > 10$.
\end{Lemma}
 
\begin{proof}
We can argue as in Lemma \ref{lemma2}, using the result of Lemma \ref{lemma3}.
\end{proof}

Now we can prove part (iii) of Theorem \ref{2emetheo}.

\begin{proof}[Proof of Theorem \ref{2emetheo} (iii)]
Suppose that $|S^2| = 10$. 
Write, as usual, $S = \{x_1, x_2, x_3, x_4\}$, $x_1 < x_2 < x_3 < x_4$, $T = \{x_1, x_2, x_3\}$ 
and $V = \{x_2, x_3, x_4\}$. 

If $S = A \dot \cup \{y\}$ with $\langle A \rangle$ abelian, then 
by Proposition \ref{prop2},  either (i) of that proposition holds,
in contradiction to Proposition \ref{prop8}, or one of (ii) and (iii) holds, in which case either
$c^y=c^2$ or $(c^2)^y=c$, which as shown in the proof of Proposition \ref{prop9},
is impossible in a polycyclic torsion-free group.
Therefore we may assume that each triple of elements in $S$ generates a non-abelian group.
In particular, $\langle T \rangle$ and $\langle V \rangle$ are non-abelian.
By Proposition \ref{prop1} $|T^2|\leq 7$ and since $\langle T \rangle$ is non-abelian, Theorem D
implies that $|T^2|= 7$.

Therefore $|S^2|=|T^2|+3$ and $\langle T \rangle$ is non-abelian.  The elements 
$x_3x_4, x_4^2$ are not in $T^2$ 
because of the ordering, 
so one of the elements $x_1x_4, x_2x_4$ belongs to $T^2$. 
Hence $x_4 \in \langle T \rangle$ and $G = \langle T \rangle$. Arguing similarly, it follows
that $|V^2| = 7$, $\langle V \rangle$ is non-abelian 
and $G=\langle V \rangle$. Therefore $T$ and $V$ satisfy the hypotheses of Proposition \ref{prop9}.

If $1 \in T$, then $1 = x_1$ by Proposition \ref{prop5}, and $V$ satisfies either (ii) or (iii) or (iv) 
of Proposition \ref{prop9}. If $V$ satisfies (ii), then $|V \cap G^{\prime}|=2$ and 
$|S \cap G^{\prime}|=3$, a contradiction, since  $ G^{\prime}$ is abelian.
If $V$ satisfies either (iii) or (iv) of Proposition \ref{prop9}, then 
$S^2 = V^2 \dot \cup V \dot \cup \{1\}$ and $|S^2| = 11$, a contradiction. Therefore we may
assume  that $1 \notin T$ and $T$ satisfies either (ii) or (iii) or (iv) of Proposition \ref{prop9}.

If $T$ satisfies (ii) of Proposition \ref{prop9}, then $|T \cap G^{\prime}|=2$  and hence 
$|V\cap G^{\prime}|\geq 1$.
Thus it follows by Proposition \ref{prop9} that $|V\cap G^{\prime}|=2$. If $|T \cap G^{\prime}|=\{x_2,x_3\}$,
then $x_2x_3=x_3x_2$ and Proposition \ref{prop5} implies that either $\langle T \rangle$ or $\langle V \rangle$
is abelian, a contradiction. Hence $x_1\in T \cap G^{\prime}$ and $|S \cap G^{\prime}|=3$,
yielding again a contradiction. 
Consequently $T$ satisfies either (iii) or (iv) of Proposition \ref{prop9} and thus it is equal to 
one of the following sets: $\{t, tc, tc^t\}$,   $\{t^{-1}, t^{-1}c, t^{-1}c^t\}$,
$\{t, tc, tc^{t^2}\}$ and $\{t^{-1}, t^{-1}c, t^{-1}c^{t^2}\}$, where
$c \in G^{\prime}$,  $c > 1$ and $c^{t^2} = cc^t = c^tc$. It follows by Lemmas \ref{lemma1},  \ref{lemma2},  \ref{lemma3} 
and \ref{lemma4} that $T$ must be equal to one of the first two sets and $S$ is as required.

A direct calculation proves the converse.
\end{proof}

Finally, we study the case $S \subseteq G$, $|S| = 5$ in the next proposition.

\begin{Proposition}
\label{prop11}
Let $S \subseteq G$, with $\langle S \rangle = G$ and  suppose that $|S| = 5$. 
Then $|S^2| > 13 = 3|S|-2$ 
\end{Proposition}

\begin{proof} 
If $|S|=5$, then by Theorem D, $|S|\geq 13$. So it suffices to assume that $|S|=13$ and
to reach a contradiction.

Write $S = \{x_1, x_2, x_3, x_4, x_5\}$, $x_1 < x_2 < x_3 < x_4<x_5$, 
$T = \{x_1, x_2, x_3, x_4\}$ and suppose that $|S^2| = 13$. Arguing as in 
the previous proposition, 
we may conclude that $|T^2| = 10$, $|S^2|=|T^2|+3$ and $\langle T \rangle = G$. Hence $T$ satisfies 
the hypotheses of part (iii) of Theorem \ref{2emetheo}.  

Suppose first that $T = \{ t, tc, tc^t, tc^{t^2}\}$, with $c > 1$ and  $c^{t^2} = cc^t = c^tc.$ 
Then $x_1 = t$, $x_2 = tc$, $x_3 = tc^t$, $x_4 = tc^{t^2}$ and as shown in the proof of Lemma \ref{lemma1},
$c,c^t\in G^{\prime}$ and $t\notin G^{\prime}$. Moreover, 
$$
T^2 = \{t^2, t^2c, t^2c^t, t^2cc^t, t^2c^tc^t, t^2cc^tc^t, t^2c^2c^t, t^2c^2c^tc^t, t^2c(c^t)^3,
t^2c^2(c^t)^3\},
$$ 
and in particular
\begin{equation}
\label{star1} 
\text{ if } \quad t^2c^\alpha(c^t)^\beta \in  T^2,\quad \text{ then }\quad \alpha \in \{0, 1, 2\}, 
\beta \in \{0, 1, 2, 3\}.
\end{equation}

Since $x_4x_5\notin T^2$, at least one of the elements $x_1x_5, x_2x_5, x_3x_5$ belongs to 
$T^2$, implying that 
$$x_5 = tc^l(c^t)^m$$  for some integers
$l$ and $m$. 
We have : 

$$
x_3x_5 = t^2c^{l+1}(c^t)^{1+m},\quad x_5x_3 = t^2c^{m}(c^t)^{l+m+1},\quad x_2x_5 = t^2c^{l}(c^t)^{1+m},
$$  
$$
x_5x_2 = t^2c^{m+1}(c^t)^{l+m},\quad x_1x_5 = t^2c^{l}(c^t)^{m}  \quad \text{ and } \quad
x_5x_1 = t^2c^{m}(c^t)^{l+m}.
$$

If each of $x_3x_5, x_5x_3, x_2x_5, x_5x_2, x_1x_5, x_5x_1$ belongs to $T^2$, 
then $l, m \in \{0, 1\}$, since 
these elements involve $\{c^l,c^m,c^{l+1},c^{m+1}\}$ and each element of $T^2$
involves only one of $\{c^0,c,c^2\}$. Hence either $x_5 = t$, or $x_5 = tc$, 
or $x_5 = tc^t$, or $x_5 = tcc^t = tc^{t^2}$, a contradiction. Therefore there exists 
$ i \in \{1, 2, 3\}$ such that either $x_ix_5 \notin T^2$ or $x_5x_i \notin T^2$. 

Now $x_5x_4 = t^2c^{m+1}(c^t)^{l+m+1}$ and $x_4x_5 = t^2c^{l+1}(c^t)^{m+2}$. 
Thus $x_5x_4=x_5x_4$ is impossible, since it implies that $l=m=1$ and $x_5=tcc^t=tc^{t^2}$, 
a contradiction.
Hence $$S^2 = T^2 \dot \cup \{x_4x_5, x_5x_4,x_5^2\}.$$
If $x_ix_5 \notin T^2$, then the only possibility is $x_ix_5 = x_5x_4$ and if 
$x_5x_i \notin T^2$, then the only possibility is $x_5x_i = x_4x_5$. 

If $x_3x_5 = x_5x_4$, then $l = m = 0$, a contradiction. If $x_2x_5 = x_5x_4$, 
then $l = 0, m = -1$, 
so $x_5 = t(c^t)^{-1} < t = x_1$, also impossible. If $x_1x_5 = x_5x_4$, 
then $l = -1, m = -2$, 
so $x_5 = tc^{-1}(c^t)^{-2} < t = x_1$, a contradiction. If $x_5x_1 = x_4x_5$, 
then $l = 2, m = 3$, 
so $x_5 = tc^2(c^t)^3$ and $x_2x_5 = t^2c^2(c^t)^4 \notin T^2$, leading us to a previous case. 
If $x_5x_2 = x_4x_5$, then $l = 2, m = 2$, so $x_5 = tc^2(c^t)^2$ and 
$x_3x_5 = t^2c^3(c^t)^3 \notin T^2$, leading us again to a previous case. Finally, 
if $x_5x_3 = x_4x_5$, then $l = 1, m = 2$, so $x_5 = tc(c^t)^2$ and
$x_5x_2 = t^2c^3(c^t)^3 \notin T^2$, and again we are in a previous case.

We have reached a contradiction in all cases. Hence if
$T = \{ t, tc, tc^t, tc^{t^2}\}$, with $c > 1$ and  $c^{t^2} = cc^t = c^tc$,
then $|S^2|>13$.
 
Now suppose that $T =  \{t^{-1}, t^{-1}c, t^{-1}c^t, t^{-1}c^{t^2}\}$ 
with $c > 1$ and $c^{t^2} = cc^t = c^tc.$ 

Arguing as before we may write $x_5 = t^{-1}c^l(c^t)^m$, for some integers $l, m$. 
Write 
$$
\bar{S} = \{t, tc, tc^t, tc^{t^2}, tc^l(c^t)^m\}.
$$ 
Then 
\begin{eqnarray*}
|\bar{S^2}| & = & |t \{1, c, c^t, c^{t^2}, c^l(c^t)^m\} t \{1, c, c^t, c^{t^2},  c^l(c^t)^m\}| \\
&=& |t^2\{1, c, c^t, c^{t^2},  c^l(c^t)^m\}^t \{1,c,c^t, c^{t^2},  c^l(c^t)^m\}| \\
&=& |\{1, c, c^t, c^{t^2},  c^l(c^t)^m\} \{1,c,c^t, c^{t^2},  c^l(c^t)^m\}^{t^{-1}}| \\
&=& |\{1, c, c^t, c^{t^2},  c^l(c^t)^m\}^{t^{-1}} \{1,c,c^t, c^{t^2},  c^l(c^t)^m \}| \\
&=& |t^{-2} \{1, c, c^t, c^{t^2},  c^l(c^t)^m\}^{t^{-1}}\{1,c,c^t, c^{t^2},  c^l(c^t)^m\}| \\
&=& |S^2|.
\end{eqnarray*}

But $|\bar{S}^2| > 13$ by the previous paragraph, so $|S^2| > 13$, as required.
\end{proof}
  
Now we can conclude the proof of Theorem \ref{2emetheo}.

\begin{proof}[Proof of Theorem \ref{2emetheo}, part (ii)]
Suppose that $|S| = k \geq 5$ and  write 
$$
S = \{x_1, x_2, \dots, x_{k-1}, x_k\}\quad \text{ and }\quad T = \{x_1, x_2, \dots, x_{k-1} \}.
$$  
Then, by Proposition \ref{prop1}, $|T^2| \leq |S^2|-3 = 3|T|-2$. If $|T^2| \leq 3|T|-3$, 
then, by Theorem D, $\langle T \rangle$ is abelian and Proposition \ref{prop2} yields a contradiction,
since case (i) of Proposition \ref{prop2} is impossible as $Z(G)=\{1\}$
by Proposition \ref{prop8} and as shown in the proof of Proposition 9,
cases (ii) and (iii) of Proposition \ref{prop2} are impossible
in a polycyclic torsion-free group.

Thus $|T^2| = 3|T|-2=|S^2|-3$ and since $k-1\geq 4$, at least one of 
the $k-1$ elements $x_1x_k, x_2x_k, \dots, x_{k-1}x_k$ belongs to $ T^2$. Hence 
$x_k \in \langle T \rangle$ and $\langle T \rangle = \langle S \rangle =G$. Thus, 
by induction, we may assume that $|S|= 5$, in which case we get the 
contradiction $|S^2| > 13$ by Proposition \ref{prop11}. Therefore $|S| \leq 4$ and 
the structure of $S$ if $|S| = 4$ follows from part (iii) of Theorem \ref{2emetheo}.
\end{proof}

\section{Subsets of the Baumslag-Solitar group $BS(1,2)$} 
\label{sec4}

Our aim in this section is to prove the following theorem.

\begin{Thm}
\label{3emetheo}
Let $G = \langle a, b \ | \ a^{b} = a^2 \rangle$ and let  
$S$ be a generating subset of $G$. 
Then the following statements hold:
\begin{itemize}
\item[(i)] If $|S| > 4$, then $|S^2| = 3|S|-2$ if and only if 
$S = \{x, xc, \dots, xc^{k-1}\}$, where $c>1$ and either $c^x = c^2$ or $(c^2)^x = c$.
\item[(ii)] If $|S| = 4$, then $|S^2| = 3|S|-2$ if and only if either 
$ S = \{1, c, c^2, x\}$, where either $c^{x} = c^2$ or  $(c^2)^x = c$,
or $ S = \{x, xc, xc^2, xc^3\}$,
where $c>1$ and either $c^{x} = c^2$  or $(c^2)^x = c.$
\end{itemize}
\end{Thm}

Throughout this section we shall denote by $G$ the Baumslag-Solitar group 
$$
G = BS(1,2) =  \langle a, b \ | \ a^{b} = a^2\rangle .
$$  
We begin with some basic well-known results concerning $G$.

\begin{Proposition} 
\label{prop12}
The following statements hold:
\begin{itemize}
\item[(i)] $G^{\prime} = \langle  a^{b^n} \ | \  n \in \mathbb{Z}\rangle$ is abelian,
\item[(ii)] $G=G^{\prime} \rtimes \langle b\rangle$,
\item[(iii)] $G$ is orderable,
\item[(iv)] If $c\in G^{\prime}$, then $c^b=c^2$,
\item[(v)] $C_{G^{\prime}}(g) = \{1\}$ for any $g \in G \setminus G^{\prime}$. 
In particular $Z(G) = \{1\}$.
\end{itemize}
\end{Proposition}

\begin{proof}
Clearly $G^{\prime}=a^G=\langle  a^{b^n} \ | \  n \in \mathbb{Z}\rangle $.
For claims (i)-(iv), see Theorem 10 in \cite{FHLMS3} and its proof.

Concerning (v), let $g \in G \setminus G^{\prime}$, $c\in G^{\prime}$ and
suppose that $c^g=c$. By (i) $g=db^s$ for some $s\in \mathbb{Z}\setminus \{0\}$ and 
$c^g=c$ implies that $c^{b^s}=c$. We may assume that  $s\geq 1$ and by (iv) we obtain
$c^{2^s}=c$. Since $2^s>1$, it follows that $c=1$, as required.
\end{proof}
 
Now, let  $\leq$ be a total order in $G$ such that $(G, \leq)$ is an ordered group. 
Let $S = \{x_1, x_2, \dots, x_k\}$ be a subset of $G$ of order $k$ and suppose 
that $x_1 < x_2 < \cdots < x_k$. We study the structure of $S$ if $|S^2| = 3|S|-2.$ 
We begin with the case $k = 3$.

\begin{Proposition} 
\label{prop13}
Let $S \subseteq G$, with $\langle S \rangle = G$ and $|S| = 3$. 
Then $|S^2| = 7$ if and only if one of the following holds:
\begin{itemize}
\item[(i)] $S = \{1, x,  y\}$, with $xy \not= yx$,
\item[(ii)] $S = \{c, c^t, t\}$, with $cc^t = c^tc$ and $c, c^t \in G^{\prime}$,
\item[(iii)] $S = \{x, xc, xc^2\}$, where $c > 1$ and either $c^{x} = c^2$ or $(c^2)^x = c.$
\end{itemize}
Furthermore, in any case, either $S\cap G^{\prime}\subseteq \{1\}$ or $|S\cap G^{\prime}|=2$.
\end{Proposition}

\begin{proof}
Suppose that $S = \{x_1, x_2, x_3\}$ and $|S^2| = 7$. If either $x_1x_2 = x_2x_1$
or  $x_2x_3 = x_3x_2$, then either (i) or (ii) holds by Proposition \ref{prop3}, since $Z(G) = \{1\}$. 

If $x_1x_2 \not= x_2x_1$ and $x_2x_3 \not= x_3x_2$, then (iii) holds 
by Proposition \ref{prop4}. In fact, cases (i) and (ii) of Proposition \ref{prop4} cannot occur, since the
relation $c^{x^2} = cc^x$ with $c > 1$ is impossible in the group $BS(1,2)$. 
For, if $c^{x^2} = cc^x$, then $c^x$ and hence $c$ belong to $G^{\prime}$ and by Proposition \ref{prop12} (iv) 
$c^b = c^2$. Clearly $x\notin G^{\prime}$, so we may assume that $x=b^s$ for some
$s\in \mathbb{Z}\setminus \{0\}$. If $x = b^s$ with $s \geq 1$, then $c^x = c^{b^s} = c^{2^s}$, 
$c^{x^2} = c^{4^s}$ and since $c^{x^2} = c^xc$, we get $c^{4^s} = c^{2^s+1}$. But $c>1$, 
so $4^s = 2^s+1$, a contradiction. Similarly, if  $x = b^{-s}$ 
with $s > 0$, then $c^{x^2} = cc^x$ implies 
$c = c^{x^{-2}}c^{x^{-1}} = c^{4^s}c^{2^s}$ and since $c>1$, we get
the contradiction $1 = 4^s+2^s$. So it follows that one of (i), (ii) or (iii) must hold.

In particular, if (i) holds, then either  $S \cap G^{\prime}=\{1\}$ or $S \cap G^{\prime}=\{1,z\}$
with $z\in \{x,y\}$, since $G^{\prime}$ is abelian, and $|S\cap G^{\prime}|=2$, as required. 

In case (ii), $c^t=cc^tc^{-1}=cd$ with
$d\in G^{\prime}\setminus \{1\}$ and $cc^t = c^tc$ implies that $cd=dc$. Hence, by 
Proposition \ref{prop12} (v), $c\in S \cap G^{\prime}$ and since $c^t\neq c$, it follows that
$t\notin  S\cap G^{\prime}$. Thus $S \cap G^{\prime}=\{c,c^t\}$ and 
$|S\cap G^{\prime}|=2$, as required.

Finally, if case (iii) holds, then $c\in G^{\prime} $ 
and $x\notin G^{\prime}$ since $[c,x]\neq 1$. Hence $S \cap G^{\prime}=\emptyset$,
as required.

A direct calculation proves the converse of the main statement.
\end{proof}

Now we study the structure of $S$ if $k \geq 4$. 
We begin with the case $S = A \dot \cup \{y\}$, where $\langle A \rangle $ is abelian.

\begin{Proposition} 
\label{prop14}
Let $S \subseteq G$, with $\langle S \rangle = G$. Suppose that
$|S| = k \geq 4$ and $S = A \dot \cup \{y\}$ with $ \langle A \rangle$ abelian. 
Then $|S^2| = 3|S|-2 $ 
if and only if  $k = 4$ and $S = \{1, c, c^2, y\}$, with either $c^y = c^2$   
or $(c^2)^y = c$.
\end{Proposition}
 
\begin{proof}
Suppose that $|S^2| = 3|S|-2$. Since $Z(G)=\{1\}$, $\langle S \rangle = G$ is not 
a nilpotent group of class at most $2$ and by Proposition \ref{prop2} we have $|S| = 4$, 
$S = \{x, xc, xc^2, y\}$ and either (ii) or (iii) of Proposition \ref{prop2} holds. 
If (iii) holds, then $x\in Z(G)=\{1\}$ and $x=1$. Thus $S$ has the required structure in this case. 

Now suppose that (ii) of Proposition \ref{prop2} holds. Then $c \in G^{\prime}$ and by Proposition \ref{prop12} (v)
$x \in C_G(c) \subseteq G^{\prime}$. Moreover, $y\notin  G^{\prime}$ since $G^{\prime}$
is abelian. Hence, by Proposition  \ref{prop12} (ii), $y=eb^s$ for some $s\in \mathbb {Z}\setminus \{0\}$
and $e\in G^{\prime}$.

If $c^y = c^2$, then $c^{b^s}=c^2$ and since $c^b=c^2$ by Proposition  \ref{prop12} (iv), 
we must
have $s=1$. Thus $x^y=x^b=x^2$ by  Proposition  \ref{prop12}  (iv) and $c^2=[y,x]=(x^{-1})^yx=x^{-2}x=x^{-1}$.
Write $d=c^{-1}$. Then $\{x,xc,xc^2,y\}=\{d^2,d,1,y\}$ with $d^y=d^2$, as required.

If $(c^2)^y=c$, then it follows that $(c^2)^{b^s}=c$ and the only possibility for $s$
is $s=-1$.
Now $x^b=x^2$, so $(x^2)^y=(x^2)^{b^{-1}}=x$ and $(c^2)^y=c=[x,y]=x^{-1}x^y$.
Thus $x^y=xc$ and $x=(x^2)^y=xcxc=x^2c^2$, yielding $x=c^{-2}$. Write $d=c^{-1}$.
Then $\{x,xc,xc^2,y\}=\{d^2,d,1,y\} $  with $(d^2)^y=(d^2)^{b^{-1}}=d$, as required.

A direct calculation proves the converse.
\end{proof}

We can now describe the structure of $S$ if $k = 4$.

\begin{Proposition}
\label{prop15}
Let $S \subseteq G$, with $\langle S \rangle = G$ and suppose that $|S| = 4$. 
Then $|S^2| = 10$ if and only if one of the following holds:
\begin{itemize}
\item[(i)]  $S = \{1, c, c^2, y\},$ with either $c^y = c^2$ or $(c^2)^y = c$,
\item[(ii)]  $S = \{x, xc, xc^2, xc^{3}\},$  with $ c > 1$ and either $ c^{x} = c^2$ or $(c^2)^x = c.$
\end{itemize}
\end{Proposition}

\begin{proof} Suppose that $S = \{x_1, x_2, x_3, x_4\}$ with $x_1 < x_2 < x_3 <x_4$
and $|S^2| = 10$. Let $T = \{x_1, x_2, x_3\}$ and $V = \{x_2, x_3, x_4\}$. If 
$S = A \dot \cup \{y\}$ with $\langle A \rangle$ abelian, then (i) holds by Proposition \ref{prop14}. 
Therefore we may assume that each triple of elements in $S$ generates a non-abelian group.
In particular, $\langle T \rangle$ and $\langle V \rangle$ are non-abelian. 

By Proposition \ref{prop1} $|T^2|\leq 7$ and since $\langle T \rangle$ is non-abelian, Theorem D
implies that $|T^2|= 7$. The elements $x_3x_4, x_4^2$ are not in $T^2$
because of the ordering,
so one of the elements $x_1x_4, x_2x_4$ belongs to $T^2$, since $|S^2| = |T^2|+ 3$.
Hence $x_4 \in \langle T \rangle$ and $G = \langle T \rangle$. Arguing similarly, it follows
that $|V^2| = 7$
and $\langle V \rangle = G$. Therefore $T$ and $V$ satisfy the hypotheses of Proposition \ref{prop13}.

If $1\in T$, then $1=x_1$ by Proposition \ref{prop5} and $V$ satisfies either (ii) or (iii) of 
Proposition \ref{prop13}. But then $S^2 = V^2 \dot \cup V \dot \cup \{1\}$ and $|S^2| = 11$, a
contradiction. Therefore we may assume that $1\notin T$. 

Thus, by Proposition \ref{prop13}, $T$ satisfies either (ii) or (iii) of that proposition. If
$|T\cap G^{\prime}|=2$, then also $|V\cap G^{\prime}|=2$ 
and $V\cap G^{\prime}\neq \{x_2,x_3\}$ by Proposition \ref{prop5}. 
Hence $|S\cap G^{\prime}|=3$, which is impossible, since
$\langle S\cap G^{\prime}\rangle $ is abelian. Therefore 
$T\cap G^{\prime}$ is an empty set and 
$T=\{x, xc, xc^2\},$  with $ c > 1$ and either $ c^{x} = c^2$ or $(c^2)^x = c$.

Suppose, first, that $c^x = c^2$. Then 
$T^2 = \{x^2, x^2c, x^2c^2, x^2c^3, x^2c^4, x^2c^5, x^2c^6\}$. Obviously $x_4^2 \notin T^2$, 
because of the ordering. Since $|S^2| = |T^2|+3$, it follows that one of the elements 
$xx_4, xcx_4, xc^2x_4$ belongs to $T^2$. Therefore $x_4 = xc^s$ for some integer $s$. 
Similarly one of the elements $xc^sx = x^2c^{2s}, xc^sxc = x^2c^{{2s}+1}, 
xc^sxc^2 = x^2c^{2s+2}$ belongs to $T^2$. Since $x_4=xc^s>xc^2$ and $c>1$, we must have 
$s\geq 3$  and since $x^2c^{2s}\leq x^2c^6$, the only possible case is that $s = 3$. 
Thus (ii) holds.

Now suppose that $(c^2)^x = c.$ In this case 
$$
T^2 = \{x^2, x^2c, x^2c^2, x^2c^x, x^2c^xc, x^2c^xc^2, x^2c^3\}
$$ 
and, as before, 
one of the elements $x_4x, x_4xc, x_4xc^2$ belongs to $ T^2$. Thus 
$x_4xc^l = xc^ixc^j$ for some integers $i,j,l$ and 
$x_4 = xc^ixc^{j-l}x^{-1} = xc^{i+2(j-l)} = xc^s$ for some integer $s\geq 3$, since $x_4>xc^2$. 

Similarly one of the elements $xx_4, xcx_4, xc^2x_4$ belongs to $T^2$. 
If $xc^2xc^s \in T^2$, then $x^2c^{s+1} \in T^2$  and $s\leq 2$, a contradiction, and 
if $xcxc^s \in T^2$, then $x^2c^xc^s \in T^2$ again implying that $s\leq 2$, 
a contradiction. Therefore the only possibile case is that $x^2c^s \in T^2$, forcing $s = 3$ and 
yielding (ii) again. Thus either (i) or (ii) holds in all cases.

A direct calculation proves the converse.
\end{proof}

Now we can prove Theorem \ref{3emetheo}.

\begin{proof}[Proof of Theorem \ref{3emetheo}]
Suppose that $|S| = k \geq 4$ and $  |S^2| = 3k-2$. If $k = 4$, then $S$ has the required structure 
by Proposition \ref{prop15}. So suppose that $k > 4. $

Write $S = \{x_1, x_2, \dots, x_k\}$, with $x_1< x_2 < \cdots <x_k$, 
$T = \{x_1, \dots, x_{k-1}\}$ and $V = \{x_2, \dots, x_k\}$. 
Since $k\geq 5$, it follows by Proposition \ref{prop14} that each  subset of $S$ with $k-1$
elements generates a non-abelian group. In particular, $\langle T \rangle$  
is non-abelian and by Theorem D  $|T^2| \geq 3|T|-2$. But by  Proposition \ref{prop1}
$|T^2| \leq |S^2| - 3 = 3|T|-2$, so it follows that  $|T^2| = 3|T|-2=|S^2| - 3$. 
Since  $k-1 \geq 4 $,
one of the $k-1$ elements $x_1x_k, x_2x_k, \dots, x_{k-1}x_k$ must belong to $T^2$.
Hence $x_k \in \langle T \rangle$ and $G = \langle T \rangle$. Similar arguments yield 
$|V^2| = 3|V|-2$ and $\langle V \rangle = G.$ 
It follows by induction that either $T = \{x, xc, \dots, xc^{k-2}\}$ 
with $c > 1$ and either $c^x = c^2$ or $(c^2)^x = c$, or  $|T| = 4$ and 
$T = \{1, c, c^2, y\}$ with either $c^y = c^2$ or $(c^2)^y = c$. 

First suppose that the latter case holds. Then $c\in G^{\prime}$, $|T \cap G^{\prime}| = 3$ and 
$|V \cap G^{\prime}| \geq 2$. But induction applied to $V$ implies that
$|V \cap G^{\prime} |\in \{0,3\}$, so $|V \cap G^{\prime}| = 3$. 
If $T \cap G^{\prime} \neq V \cap G^{\prime}$, then $|S \cap G^{\prime}|=4=k-1$
and we have reached a contradiction, since $G^{\prime}$ is abelian.
Hence $T \cap G^{\prime} = V \cap G^{\prime}=\{x_2, x_3, x_4\}$.
Now, if one of the elements $x_1x_2, x_1x_3, x_1x_4$ belongs to $\{x_2, x_3, x_4\}^2$, 
then $\langle T \rangle$ is abelian and if one of the elements $x_2x_5, x_3x_5, x_4x_5$ 
belongs to $\{x_2, x_3, x_4\}^2$, then $\langle V \rangle$ is abelian, 
a contradiction in both cases.  So neither of the above six elements belongs 
to  $\{x_2, x_3, x_4\}^2$ and  since $x_1x_2 < x_1x_3<x_1x_4<x_2x_5<x_3x_5<x_4x_5$, it 
follows that $|S^2| \geq 6+|\{x_2, x_3, x_4\}^2|\geq 6+5=11$, a contradiction.
 
So we may assume that $T = \{x, xc, \dots, xc^{k-2}\}$, with $c > 1$
and either $c^x = c^2$ or $(c^2)^x = c$. Thus $x_i=xc^{i-1}$ 
for $i=1,2,\dots,k-1$.

First assume that $c^x = c^2$. 
Since $|T^2| = |S^2|-3$ and $k-1\geq 4$, there exists an integer $i$, $1\leq i\leq k-1$, such that 
$xc^ix_k \in T^2$. Hence $xc^ix_k = xc^jxc^l$
for some integers $i, j, l$ and 
$x_k = c^{j-i}xc^l=xc^{2(j-i)+l}=xc^s$ for an integer $s$. Now either 
$x_kx_{k-2} \in T^2$ or $x_kx_{k-2} \notin T^2$. 

If $x_kx_{k-2} \in T^2$, then  $xc^sxc^{k-3} = xc^rxc^{k-2}$ for some integer $r$
and $x^2c^{2s+k-3}=x^2c^{2r+k-2}$. 
Thus $2s+k-3 = 2r+k-2$, which is impossible, since these
numbers are of different parity. 

So suppose that $x_kx_{k-2} \notin T^2$. Then $S^2\setminus T^2=\{x_kx_{k-2}, x_kx_{k-1},x_k^2\}$
and since $x_{k-1}x_k\in S^2\setminus T^2$, it follows that either $x_{k-1}x_k=x_kx_{k-1}$
or $x_{k-1}x_k=x_kx_{k-2}$. If $x_{k-1}x_k=x_kx_{k-1}$, then $x^2c^{2(k-2)+s}=x^2c^{2s+k-2}$
and $2(k-2)+s=2s+k-2$, yielding $s=k-2$ and $x_k\in T$, a contradiction.
Hence $x_kx_{k-2} = x_{k-1}x_k$, yielding $2s+k-3=2(k-2)+s$. Thus $s = k-1$, as required.

Now, suppose that $(c^2)^x = c$. Then $c^{x^{-1}}=c^2$ and arguing as in the previous case,
there exists $i$  such that $1\leq i\leq k-1$
and $x_kxc^i \in T^2$. Hence $x_kxc^i= xc^jxc^l$ for some integers $i, j, l$
and
$x_k = xc^jxc^{l-i}x^{-1} = xc^j(c^{l-i})^{x^{-1}}=xc^s$ for some integer $s$. 
Now either 
$xc^{k-3}xc^s \in T^2$ or $xc^{k-3}xc^s \notin T^2$. 

If $xc^{k-3}xc^s \in T^2$, then  
$xc^{k-3}xc^s = xc^{k-2}xc^r$ for some integer $r$ and $xc^s = cxc^r$. Thus 
$c^{2s}x=c^{1+2r}x$, which is impossible. 

Thus $xc^{k-3}xc^s \notin T^2$. Then $S^2\setminus T^2=\{x_{k-2}x_k, x_{k-1}x_k,x_k^2\}$ 
and since $x_kx_{k-1}\in S^2\setminus T^2$, it follows that either $x_kx_{k-1}=x_{k-1}x_k$
or $x_kx_{k-1}=x_{k-2}x_k$. If $x_kx_{k-1}=x_{k-1}x_k$, then $c^sxc^{k-2}=c^{k-2}xc^s$
and $s+2(k-2)=k-2+2s$. Thus $s=k-2$ and $x_k\in T$, a contradiction.
Hence $x_kx_{k-1}=x_{k-2}x_k$, yielding $s+2(k-2)=k-3+2s$. Thus $s = k-1$, as required.
The proof in one direction is complete.

Conversely, suppose that $|S|=k$. If $k = 4$ and $S$ satisfies the corresponding conditions, 
then $|S^2| = 10$ by Proposition \ref{prop15}. If $k>4$ and 
$S = \{x, xc, \dots, xc^{k-1}\}$ with $c^x = c^2$, then  
$$
S^2=\{xc^ixc^j=x^2c^{2i+j}\ \mid \ 0\leq i,j\leq k-1\}.
$$ 

We claim that $|S^2|=3k-2$. Since $0\leq 2i+j\leq 3k-3$, it suffices to show that
each integer $n$ with $0\leq n\leq 3k-3$ can be represented in the form $n=2i+j$
for some $ 0\leq i,j\leq k-1$. If $n=0$, then $n=2\cdot 0+0$ and if $n>0$, then
$n=3a+b$, where $0\leq a\leq k-2$ and $1\leq b\leq 3$.
In the latter case, the required representations are: $n=2a+(a+1)$ if $b=1$, 
$n=2(a+1)+a$ if $b=2$ and $n=2(a+1)+(a+1)$ if $b=3$. 

Hence $|S^2|=3|S|-2$, as required. If 
$S = \{x, xc, \dots, xc^{k-1}\}$ with $(c^2)^x = c$, then $c^{x^{-1}} = c^2$ and
$x^{-1}S^2x=\{c^ixc^jx=c^{i+2j}x^2\ \mid \ 0\leq i,j\leq k-1\}$. As shown above 
$|S^2|=|x^{-1}S^2x|=3|S|-2$, as required.
\end{proof}

\section{Proof of Theorem \ref{1ertheo}}
\label{sec5}
 
Now we can prove Theorem \ref{1ertheo}.

\begin{proof}[Proof of Theorem \ref{1ertheo}]
Let $S \subseteq G$, $|S| = k \geq 4$ and suppose that $|S^2| = 3k-2$ and $\langle S \rangle$
is non-abelian. Then $\langle S \rangle$ satisfies either (ii), or (iii), or (iv), or (v) of Theorem F.
From now on, (ii), (iii), (iv) and (v) denote items of Theorem F and (1i), (1ii) and (1iii) denote items
of Theorem \ref{1ertheo}.

If $G$ is a nilpotent ordered group of class $2$ and $S$ is a finite subset of $G$ 
such that $|S|=k\geq 4$ and $\langle S \rangle$ is non-abelian, then by Proposition \ref{prop6} $|S^2| = 3k-2$
if and only if $S$ satisfies (1i). This takes care of $\langle S \rangle$
satisfying (ii), since in cases (iii), (iv) and (v),  $\langle S \rangle$ is
non-nilpotent. If $\langle S \rangle$ satisfies (iii), 
then Theorem \ref{2emetheo} implies that $k=4$ and case (a) of (1iii) holds. 

Suppose, now, that  $\langle S \rangle$ satisfies (iv). Then $S$ satisfies the assumptions 
of Theorem \ref{3emetheo}. Hence, if $k>4$, then $|S^2| = 3k-2$
if and only if $S$ satisfies (1ii) and  if $k=4$, then $|S^2|=3k-2$ if and only if $S$ 
satisfies 
either (1ii) or case (b) of (1iii). Notice that if $k=4$, then the second case in item (ii)
of Theorem \ref{3emetheo} is of type (1ii).  

Finally, if $\langle S \rangle$ satisfies (v), then $|S| = 4$ and $\langle S \rangle$
is not a nilpotent group of class at most $2$. Hence, by Proposition \ref{prop2}, 
$S$ satisfies one of the cases (c) or (d) of (1iii).
 
Conversely, it follows from our proof that if $S$ satisfies one of the conditions (1i), (1ii) 
or case (b) of (1iii), then $|S^2| = 3k-2$. 
\end{proof}

\section*{Acknowledgements}

This work was supported by the "National Group for Algebraic and Geometric Structures, and their Applications" (GNSAGA - INDAM), Italy.

A.P. is supported by the ANR grant Caesar number 12 - BS01 - 0011. He thanks his colleagues from the University of Salerno for their hospitality 
during the preparation of this paper.

\end{document}